\documentclass[11pt,reqno]{amsart}

\usepackage{CBstyle}
\usepackage[T1]{fontenc}
\usepackage[english]{babel}
\usepackage{booktabs}

\title[Rational points on ellipsoids]{Rational points on ellipsoids\\ and modular forms}
\author{C.~Burrin}
\address{Institute of Mathematics, University of Zurich}
\email{claire.burrin@math.uzh.ch}

\author{M.~Gr\"obner}
\address{Departement of Mathematics, ETH Zurich}
\email{matthias.groebner@student.ethz.ch}

\begin{document}

\maketitle

\begin{abstract}
The theory of modular forms and spherical harmonic analysis are applied to establish new best bounds towards the counting and equidistribution of rational points on spheres and other higher dimensional ellipsoids,  in what may be viewed as a textbook display of the `unreasonable effectiveness' of modular forms. 
  \end{abstract}

\section{Introduction}

Points with rational coordinates on the $d$-dimensional unit sphere 
$$
S^d=\{x\in \R^{d+1}: |x|^2=1\}
$$ 
form a dense countable set with a rich theory of (intrinsic\footnote{ Intrinsic refers here to the theory of approximation of points on the spheres by rational points that also lie on the sphere, and not simply in the general ambient space.}) Diophantine approximation. To study the distribution of rational points on the sphere, we typically order these by height; a rational point $x\in \Q^{r}$ has height $n\in\N$ if it can be written as $x=\tfrac{{\bf m}}{n}$ for some ${\bf m}\in \Z^{r}$ such that the greatest common divisor is $({\bf m},n)=1$. Then it is known that the sets $\Omega_n$ of all rational points of height $n$ (odd) on $S^d$ become equidistributed on the sphere as $n\to \infty$; see \cite{Duke2003} for $d=2$ and \thmref{Thm2} below for a precise statement that extends to all higher dimensional ellipsoids. Here, equidistribution is intimately connected to the arithmetic nature of the sets $\Omega_n$; the number $|\Omega_n|$ of rational points on the sphere is expressed by the arithmetic sum
$$
|\Omega_n|  
= \sum_{\delta\mid n} \mu(\delta)\, r_{d+1}(n^2/\delta^2),
$$
where $r_{d+1}(n)$ is the classical sum-of-squares function, $\mu(n)$ is the classical M\"obius function, and the sum is taken over the positive divisors of $n$. 
A common feature of such arithmetic functions is that their averages behave with striking regularity, and indeed the asymptotic growth of the set $\Omega_T$ of all rational points on the sphere of height {\em up to} $T$  is given by
 \begin{align}\label{size}
 |\Omega_T| \sim c_d\, T^d
 \end{align}
 as $T\to\infty$, for a uniform average density constant $c_d>0$. 
 
 To establish \eqref{size}, different approaches are available, including the circle method \cite{HeathBrown1996,Getz2018}, harmonic analysis on groups \cite{DukeRudnickSarnak1993,GorodnikNevo2012}, the recent study of light-cone Eisenstein series for $\G\setminus  \h^{d+1}$ \cite{KelmerYu2022,KelmerYu2023}, or analysis of modular forms \cite{Duke2003}, each offering unique insights. For instance, in \cite{KelmerYu2022}, which partly inspired this article, Kelmer and Yu make the observation that the Riemann hypothesis is true if and only if a smooth version of this counting problem holds. 
 
  We adopt the approach of modular forms to provide improvements on what is currently known on the counting and equidistribution of rational points on the sphere and higher dimensional ellipsoids, in what we view as a textbook illustration of the `unreasonable effectiveness of modular forms.' 

 \subsection*{Main results}
 Let $Q(x)$ be a positive definite quadratic form in $d+1$ variables, $d\geq2$, and let $\cE_Q=\{x\in \R^{d+1}: Q(x)=1\}$ be the associated ellipsoid. Our results and methods assume the following requirements on $Q(x)=\tfrac{1}{2}x^T Ax$:
  \begin{enumerate}
      \item The symmetric positive definite matrix $A$ has integral entries and even diagonal entries; 
      \item There exists a positive integer $N$, 
with $N$ equal to  4 times a squarefree odd integer, such that $NA^{-1}$ is integral with even diagonal entries. 
  \end{enumerate}
  In the case of the sphere, $Q(x)=|x|^2$, $A=2 I_{d+1}$ and $N=4$. We will use $\Omega_T$ to now denote the set of rational points on $\cE_Q$ of height up to~$T$, and $\Omega_n$ to denote the set of rational points on $\cE_Q$ of height equal to $n$. We express the asymptotic equidistribution of the sets $\Omega_n$ as follows.
  
    \begin{Thm}\label{Thm2}
    Let $Q$ be a positive definite quadratic form as described above. Let $\Psi\subset \cE_Q$ be a convex domain with piecewise smooth boundary, and let $\mu_Q $ denote the normalized Lebesgue measure on $\cE_Q$. Then 
    \begin{align}\label{equid_rate}
      \frac{|\Omega_n\cap \Psi|}{|\Omega _n|} = \mu_Q(\Psi) +O_{Q,\Psi,\eps}(n^{-(d-1)/(2d+3)+\eps})
    \end{align}
    as $n\to\infty$ along the sequence of integers coprime to $N$.
    \end{Thm}
    
    Our main result, as advertised, concerns the number of rational points on $\cE_Q$ of height up to $T$ with a precise error term. This is
  \begin{Thm}\label{Thm1}
    Let $Q$ be a positive definite quadratic form as described above. There exists a positive constant $c_Q$ such that 
 \begin{align}\label{record}
 |\Omega_T| = c_Q\, T^d\op{1+O_Q(T^{-d/(d+1)+\eps})}
 \end{align}
 as $T\to\infty$ for any $\eps>0$.
  \end{Thm} 

By a pigeonhole argument, the optimal error term is of order $O(T^{-1+\eps})$, whence the obtained result is close to optimal in large dimensions. Assuming that the Generalized Riemann Hypothesis (GRH) is true, the error term can be sharpened further; see \rmkref{rmk:GRH}. Moreover, an artefact of our approach is that the average density constant $c_Q$ is expressible as a computable residue; we will return to this point in Section \secref{sec:sketchofproofs}.

By averaging over \eqref{equid_rate}, we find that $\Omega_T$ becomes equidistributed on $\cE_Q$ with rate $O(T^{-(d-1)/(2d+3)+\eps})$ as $T\to\infty$. The proof of  \thmref{Thm1} can be adapted to recover a stronger rate of equidistribution when $d$ is even --- the reason for this restriction will become apparent at the end of \secref{sec:sketchofproofs}.

  \begin{Thm}\label{Thm2bis}
    Keeping the notation as above and $d$ even, we have
    \begin{align*}
      \frac{|\Omega_T\cap \Psi|}{|\Omega _T|} = \mu _Q(\Psi) + O_{Q,\Psi}(T^{-d/(2d+2)+\epsilon })
    \end{align*}
    as $T\to\infty$. 
  \end{Thm}

These rates improve on the corresponding recent results of \cite{KelmerYu2022, KelmerYu2023}. Finally we note that specializing $\Psi\subset\cE_Q$ to be a hyperspherical cap on the sphere $S^d$, equidistribution will occur at double the speed. In the sequel to this paper we  examine the small-scale distribution of rational points of height $n$ on the sphere.

\section{Discussion: Sketch of proofs and further results}\label{sec:sketchofproofs}

\subsection*{Sketch of proof}
The proof of \thmref{Thm1} builds on the study of the multiplicative generating series for $\Omega_n$, i.e.,
$$
R_Q(s) = \sum_{n=1}^\infty \frac{|\Omega_n|}{n^s}.
$$
An application of Perron's formula asserts that for any $T\not\in \N$ we have 
\begin{align*}
|\Omega_T| = \frac{1}{2\pi i} \int_{(\beta )} R_Q (s) \frac{T^s}{s}\, ds
\end{align*}
with integration in the complex plane over the vertical line passing through $\beta\in\R$, when $\beta$ lies to the right of the abscissa of absolute convergence of the series $R_Q (s)$.

Since $\Omega_n = \{\bm\in\Z^{d+1}:(\bm,n)=1,\, Q(m)=n^2\}$, we may express this generating series in more classical terms as
$$
R_Q(s) = \frac{1}{\zeta(s)}\sum_{n=1}^\infty \frac{r_Q(n^2)}{n^s},
$$
for the representation numbers $r_Q(n)= |\{\bm\in\Z^{d+1}:Q(\bm)=n\}|$.  
These representation numbers have a rich theory. In particular, they appear as the Fourier coefficients of the theta function associated to $Q$, namely
$$
\Theta_Q = \sum_{\bm\in \Z^{d+1}} q^{Q(\bm)} = \sum_{n=0}^\infty r_Q(n) q^n\qquad (q=e^{2\pi iz},\, z\in\h).
$$
Under the assumptions imposed on $Q$ earlier, $\Theta_Q$ is a modular form of weight $\tfrac{d+1}{2}$, level $N$, and Nebentypus a quadratic character modulo $N$ (see \secref{sec:appendix theta}).  We will prove in Section \secref{sec:proof of AC} that $R_Q(s)$ converges absolutely when $\re(s)>d$ and admits an analytic continuation to $\re(s)\geq \tfrac{d}{2}$ except for possible simple poles at $s=d$ and $s=\tfrac{d+1}{2}$. This is where the analysis of modular forms comes into play; see \thmref{thm:modular decomposition} below for a discussion. For now, we choose $\beta =d+\epsilon$ for some small $\epsilon >0$, and truncating Perron's formula at height $H$ we obtain that
\begin{align*}
|\Omega_T| = \frac{1}{2\pi i} \int_{\beta -iH}^{\beta +iH} R_Q (s) \frac{T^s}{s}\, ds + O\op{\frac{T^\beta }{H}}.
\end{align*}
A careful application of Cauchy's residue theorem --- requiring integral mean square estimates for standard $L$-functions of degrees 1 to 3 --- then yields \thmref{Thm1} with leading constant given by the residue  
\begin{align*}
    c_Q = \underset{s=d}{\rm Res}\, \frac{R_Q (s)}{s}.
  \end{align*}

\subsection*{Input from the theory of modular forms and further results}\label{sec:results modular forms}

Via Weyl's criterium, estimates on the size of the Fourier coefficients of theta series imply equidistribution in the setting of rational points, as well as in the related, well studied, setting of lattice points projected onto the sphere. In fact, when $d\geq3$, standard estimates imply the equidistribution of lattice points projected onto $S^d$ \cite{Pommerenke1959,GolubevaFomenko1985}. The case of $d=2$ famously requires more effort, due to the lack of a Ramanujan--Petersson bound for primitive forms of half-integer weight, and was settled by Duke \cite{Duke1988} and Golubeva and Fomenko \cite{GolubevaFomenko1987} using the work of Iwaniec \cite{Iwaniec1987}. This added difficulty is absent in the setting of rational points thanks to the Shimura correspondence (reviewed in \secref{sec:appendix Shimura}). The Shimura lift will relate Fourier coefficients of a half-integer weight primitive form at square arguments to the Fourier coefficients of an even weight primitive form, for which the Ramanujan--Petersson conjecture was settled by Deligne; see \thmref{thm:bounds on Fourier coefficients} from which we deduce \thmref{Thm2} in \secref{sec:proofs}.
 
 In the proof of \thmref{Thm1} sketched above, the main analytic input is, as announced, a specialization of the following theorem on Dirichlet series to the generating series $R_Q(s)$ (as carried out in \secref{sec:Collecting/Specializing}). It is in the course of its proof that we encounter the `unreasonable effectiveness' of modular forms.

\begin{Thm}\label{thm:modular decomposition}
  Let $k \geq 1$ be an integer or half-integer.   Given a modular form $f\in M_k(N,\chi )$ with $q$-expansion $f=\sum a(n)q^n$, the associated Dirichlet series 
  \begin{align*}
     L_f(s)=  \sum_{n=1}^\infty \frac{a(n^2)}{n^s}
  \end{align*}
converges absolutely for $\re (s)>2k-1$ and admits a meromorphic continuation, which is holomorphic in the half-plane $\re(s)\geq \tfrac{2k-1}{2}$ with the exception of possible poles at $s=2k-1$ and $s=k$.
\end{Thm}
\begin{proof}[Sketch of proof]
  If $f$ is a cusp form, the statement follows  from the analytic theory of the Shimura lift (when $d$ is even) and of the symmetric square (when $d$ is odd) \cite{Shimura1973,Shimura1975}, and $L_f$ is expressed via $L$-functions of degree 2 and 3, respectively. When $f$ is not a cusp form, we choose a preferred basis for the finite-dimensional vector space $M_k(N,\chi)$. Then $L_f$ is expressed via Dirichlet series whose coefficients are twisted sum-of-divisors functions. The statement follows from explicit $L$-series identities for these divisor sums (derived in \secref{sec:div}) and the standard theory of $L$-functions of degree 1.
\end{proof}

The stronger estimates for the equidistribution rates for $\Omega_T$  when $d$ is even reflects the current state of affairs: the analytic theory of $L$-functions of degrees 1 and 2 is better understood than that of $L$-functions of degree 3. To be more precise, evaluating $|\Omega_T|$ through Cauchy's residue theorem requires bounds on the integral mean square of various standard $L$-functions of degree 1--3, along their critical line. In general, it is conjectured that for a primitive $L$-function attached to a cuspidal automorphic representation of $\GL(n)$ there is a constant $C(\pi )$ such that
\begin{align*}
  \int_0^T |L(\tfrac{1}{2}+it,\pi)|^2 dt \sim C(\pi)\, T\log T
\end{align*}
as $T\to\infty$. This is known to be true for primitive $L$-functions of degree 1 and 2 (see, e.g., \cite{HardyLittlewood1918,Good1974,Zhang2005,Zhang2006}), but  remains open for $L$-functions of degree higher than 2. Milinovich and Turange-Butterbaugh showed that for all primitive $L$-functions of degree $n$, we have
\begin{align*}
  \int_0^T |L(\tfrac{1}{2}+it,\pi)|^2 dt \ll T(\log T)^{1+\eps}
\end{align*}
conditional on GRH and a mild conjecture towards Ramanujan--Petersson \cite{MilinovichTurnage2014}. In \secref{sym square mean square}, we give an unconditional proof of the following weaker result for the symmetric square $L$-function. 

\begin{Thm}\label{thm:sym square mean square}
Let $f\in S_k(N,\chi )$ be a primitive form for $\Gamma_0(N)$ of weight $k$ and character $\chi $ (mod $N$). Then as $T\to\infty$,
\begin{align*}
  \int_0^T |L(\tfrac{1}{2}+it,\sym^2 f)|^2 dt \ll T^{3/2} (\log T)^{18}.
\end{align*}
\end{Thm}

Our proof is based on an adaptation of Ivi\'c's approximate functional equation for the Riemann zeta function \cite{Ivic2003}.  An estimate of this strength was previously stated in \cite{Sank2002} for level 1 forms.

\subsection*{Acknowledgements}
The first named author acknowledges the support of Swiss National Science Foundation  Grant No. 201557.

\section{Modular forms and $L$-functions}\label{sec:appendix modular forms and L-functions}

For background,  
we refer the reader to \cite{CohenStromberg} and \cite[Chapter 14]{IwaniecKowalski2004}.

\subsection{Modular forms for $\G_0(N)$}

Let $\h=\{x+iy: y>0\}$ denote the upper half-plane, and let $\overline \h$ denote the extended upper half-plane, i.e., $\overline \h = \h \cup \R \cup \{\infty\}$. Throughout the text, let $q= e^{2\pi iz}$. In this paper we only work with the Hecke congruence group
\begin{align*}
  \G_0(N) = \ob{ \bpm a & b \\ c& d\epm \in \SL_2(\Z) : N\mid c},
\end{align*}
which acts on $\h$ by fractional linear transformation. For each integer $k\geq1$ this action defines the `slash' operator $f\vert_k (\bsm *&* \\ c& d \esm) (z) \coloneqq (cz+d)^{-k} f(\tfrac{az+b}{cz+d})$ on functions on $\h$. Note that $-I\in \G_0(N)$ and $f\vert_n (-I) = (-1)^k f$.

A modular form $f:\h\to\C$ of weight $k$, level $N$, and Nebentypus $\chi$ is characterized by two stringent requirements, namely that it is a holomorphic function on $\h \cup \Q \cup \{\infty\}$ and that it transforms as $f\vert_k \gamma  = \chi(\gamma ) f$ for each $\gamma  \in \Gamma_0(N)$. Note that $f=0$ unless $\chi(-1)=(-1)^k$. Modular forms of fixed level, weight and Nebentypus form a finite dimensional vector space, which we denote by $M_k(N,\chi )$. Equipped with the Petersson inner product, this space admits the orthogonal decomposition
\begin{align*}
  M_k(N,\chi ) = \cE_k(N,\chi ) \oplus S_k(N,\chi ).
\end{align*}
The subspace $\cE_k(N,\chi )$ is the space of Eisenstein series, and its orthogonal complement  $S_k(N,\chi )$ consists of those forms that vanish at all the cusps of $\G_0(N) \setminus \h$. The dimension of these spaces can be computed explicitly; we only note that  
$
   \dim S_k(N,\chi ) = O_N(k)
$.

\subsection{$L$-functions associated to modular forms}

Let $k$ be a positive integer. For any cusp form $f\in S_k(N,\chi )$ we denote its Fourier coefficients at $\infty$ by $a_f(n)$, that is, $ f= \sum_{n=1}^\infty a_f(n) q^n.$
As is customary in analytic number theory, we renormalize the Fourier coefficients of $f$ to $\lambda _f(n)\coloneqq  a_f(n)n^{-(k-1)/2}$. To each cusp form $f$ we now associate the Dirichlet series
\begin{align*}
  L(s,f) = \sum_{n=1}^\infty \frac{\lambda_f(n)}{n^s},
\end{align*}
which converges absolutely and uniformly on compact sets for sufficiently large $\re(s)$, admits a holomorphic continuation to all of $\C$, and a completed $L$-function with a functional equation, whose symmetry is about the vertical line at $\re(s)=\tfrac{1}{2}$. These properties are reminiscent of those of Dirichlet $L$-functions, with the exception that no existence of an Euler product has been mentionned so far.
For this we need to restrict $f$ to the subspace $S_k^{\rm new}(N,\chi)$ of newforms. Note that we have the direct sum decomposition
\begin{align}\label{cusp and new}
S_k(N,\chi) = \bigoplus_{M\mid N}\bigoplus_{m\mid \tfrac{N}{M}} B(m) S_k^{\rm new}(M,\chi),
\end{align}
where $(B(m)f)(z)=f(mz)$. We say that $f$ is primitive if it is a simultaneous eigenfunction of all Hecke operators and is normalized such that its first Fourier coefficient is $a_f(1)=1$. Then the Hecke $L$-function $L(s,f)$ admits the Euler product of degree 2
\begin{align*}
  L(s,f) = \prod_p (1-\lambda_f(p)p^{-s}+\chi(p)p^{-2s})^{-1}.
\end{align*}

The Rankin--Selberg convolution $L$-function attached to two primitive forms $f$, $g\in S^{\rm new}_k(N,\chi)$ is defined for sufficiently large $\re(s)$ by
\begin{align*}
  L(s,f\otimes g) = L(2s, \chi^2) \sum_{n=1}^\infty \frac{\lambda_f(n)\lambda_g(n)}{n^s}.
\end{align*}
This function admits an Euler product of degree $4$, and has a completion that analytically extends to all of $\C$ except when $f=g$, in which case there are simple poles at $s=0,1$. Furthermore, the completed function satisfies a functional equation relating $s$ to $1-s$. For the precise form of the completed $L$-function and its functional equation we refer to \cite{Li1979}. The symmetric square $L$-function associated to a primitive $f \in S^{\rm new}_k(N,\chi)$ is given by
\begin{align*}
  L(s,\sym^2 f) = L(2s, \chi ^2) \sum_{n=1}^\infty \frac{\lambda _f(n^2)}{n^s};
\end{align*}
it converges absolutely for $\re(s)>1$, admits an Euler product of degree $3$, and satisfies a functional equation linking $s$ to $1-s$.  The completed $L$-function 
\begin{align*}
  \Lambda(s,\sym^2 f) &= \pi^{-\tfrac{3}{2}(s+k-1)} \Gamma\left( \tfrac{s+k-1}{2} \right) \Gamma\left( \tfrac{s+k}{2} \right) \Gamma\left( \tfrac{s+1-\lambda}{2} \right)  L(s,\sym^2 f) ,
\end{align*}
where $\lambda$ is $0$ or $1$ according as $\chi(-1) = 1$ or $-1$, extends to a holomorphic function on all of $s \in \C$ except for possible simple poles at $s = 0,1$. Conditions under which $\Lambda(s,\sym^2 f)$ is holomorphic at $s=1$ or $s = 1$ are given in \cite[Theorem 2]{Shimura1975}. It was proved by Gelbart and Jacquet that $\Lambda(s,\sym^2 f)$ is the $L$-function of an automorphic representation of $GL(3, \mathbb{A}_\mathbb{Q})$ \cite{GelbartJacquet1978}.

\subsection{Half-integral weight modular forms} \label{sec:appendix Shimura}

Let $k\in \tfrac{1}{2}+ \N$, $N$ be divisible by 4, and $\chi$ be an even Dirichlet character modulo $N$. We say that a function $f:\h\to\C$ is a modular form of weight $k$, level $N$ and Nebentypus $\chi$ if it is holomorphic on $\h\cup \Q\cup\{\infty\}$ and transforms as $f\vert_k \gamma  = \nu_\theta^{2k} (\gamma) \chi (\gamma ) f$ for all $\gamma \in \Gamma_0(N)$. Here $\nu_\theta$ is the theta multiplier system, given by  
$$ \nu_\theta(\gamma) :=  \left(\frac{c}{d}\right) \eps_d^{-1}   \quad \text{ for }  \gamma = \bpm a & b \\ c& d\epm  \in \G_0(N), $$   
where $\eps_d=1$ or $i$ according as $d\equiv1$ or $3$ modulo $4$, and the symbol $(\tfrac{\cdot}{\cdot})$ is Shimura's extension of the Jacobi symbol, see \cite{Shimura1973}.

Half integral weight modular forms are related to even  integral weight modular forms through the Shimura correspondence \cite{Shimura1973,Niwa1975,Cipra1983}. Let $f\in S_{k}(N,\chi)$ be a half integral weight modular form as above and define a sequence $(A(n))_{n\geq1}$ via
\begin{align*}
  \sum_{n=1}^\infty \frac{A(n)}{n^s} = L(s-k+\tfrac{3}{2},\tilde \chi ) \sum_{n=1}^\infty \frac{a_f(n^2)}{n^s},
\end{align*}
with Dirichlet character  
$$
\tilde \chi(n) = \chi(n) \left( \frac{-1}{n} \right)^{k - \frac12}.
$$ 
The Shimura correspondence asserts that the lift $\tilde f = \sum A(n)q^n$ is a modular form in $M_{2k-1}(N/2,\chi^2)$ if $k\geq \tfrac{3}{2}$. If $k\geq \tfrac{5}{2}$, then $\tilde f \in S_{2k-1}(N/2,\chi^2)$. Moreover, the map $f\mapsto  \tilde f$ commutes with Hecke operators. Additionally, there is a version of the Shimura correspondence that maps a non-cusp form $f \in S_k(N,\chi )^\perp$ into $M_{2k-1}(N/2,\chi^2)$, see  \cite[Chapter 15.1.4]{CohenStromberg}. We will use this variant only when $k-\tfrac{1}{2}$ is even, giving us in such cases
$$ \frac{ a_f(0)}{2} L(\tfrac{3}{2} - k, \chi) + \sum_{n=1}^\infty A(n) q^n \in M_{2k-1}(N/2,\chi^2).$$

\subsection{Theta functions} \label{sec:appendix theta}

 The prototypical example of a modular form of half-integer weight is the classical theta series 
\begin{align*}
  \theta  = \sum_{n\in \Z} q^{n^2}
\end{align*}
with weight $1/2$, level 4 and trivial Nebentypus. More generally, we now recall that theta functions associated to quadratic forms give rise to integral and half-integral weight modular forms \cite{Schoeneberg1939,Pfetzer1953,Shimura1973}. Let $A$ be a $r \times r$ positive definite symmetric integral matrix with even diagonal entries and $Q$ the quadratic form associated with $A$, that is, $Q(x) = \frac12 x^T A x$. Let $N$ be the least positive integer such that $NA^{-1}$ has integral and even diagonal entries. A spherical function for $Q$ is a polynomial $P(x)$ in $d+1$ variables that is homogeneous and harmonic with respect to $\Delta_A=\sum_{i,j} a^*_{ij} \tfrac{\partial^2}{\partial x_i \partial x_j}$, where $A^{-1}=(a^*_{ij})$. Let $P$ be a spherical function for $Q$ of degree $\nu$.
Then the theta function associated with $P$ and $Q$, given by
\begin{align*}
  \Theta  = \sum_{{\bf m}\in \Z^{r}}P({\bf m})q^{Q({\bf m})} = \sum_{n=0}^\infty \sum_{\substack{{\bf m}\in \Z^{r} \\ Q({\bf m})=n }} P({\bf m}) q^n  = \sum_{n=0}^\infty r_\Theta (n) q^n,
\end{align*}
defines a modular form of weight $\tfrac{r}{2}+\nu$, level $N$, and Nebentypus $$
\chi = \begin{dcases}
(\tfrac{(-1)^{r/2}\det(A)}{\cdot}), &\text{ if } r \text{ is even} ;\\
(\tfrac{2\det(A)}{\cdot}), &\text{ if } r \text{ is odd}.
\end{dcases}
$$ 
It should be noted that since $P$ is a homogeneous polynomial,
\begin{itemize}
\item $\Theta=0$ if $\nu$ is odd;
\item $\Theta$ is a holomorphic cusp form if $\nu>0$;
\item and $P=1$ if and only if $\nu=0$.
\end{itemize}

\subsection{Estimates on Fourier coefficients}

We now record estimates on Fourier coefficients of modular forms at square arguments. For this we will need the twisted divisor sums      
\begin{align*}
      \sigma_k(\chi_1, \chi_2, n)  = \sum_{\delta \mid n} \chi_1(\delta) \chi_2(n/\delta) \delta^k
    \end{align*}
     where $k$ be a positive integer, $\chi_1$, $\chi_2$ are Dirichlet characters modulo $N_1$ and $N_2$, respectively. 
      It is easy to see that $\sigma_k(\chi_1,\chi_2,n)=0$ if $(n,N_1 N_2)>1$ and that 
    \begin{align}\label{eq:divasymp}
      |\sigma_k(\chi_1, \chi_2, n)| \asymp n^{k}
    \end{align} 
    whenever $(n,N_1 N_2)=1$.

\begin{thm}\label{thm:bounds on Fourier coefficients}
  Let $k\geq1$ and take $f\in M_k(N,\chi)$ with Fourier expansion at $\infty$ given by $f= \sum a(n)q^n$. Then
  \begin{enumerate}
      \item  If $f$ is primitive and $k\neq\tfrac32$, then $|a(n^2)|\ll_\eps  n^{k-1+\eps}$ for every $n\geq1$; 
      \item If $f$ is cuspidal but not primitive and $k\neq\tfrac32$, we have  
  \begin{align*}
      |a(n^2)| \ll_\eps \op{\dim(S_k(N,\chi ))N^{\epsilon -1}\frac{(4\pi )^k k^\epsilon}{\Gamma (k)} \scal{f,f}}^{1/2}  n^{k-1+\epsilon },
  \end{align*}
where $\scal{f,f}$ is the Petersson norm of $f$; 
\item If $f$ is not cuspidal then $|a(n^2)|\gg n^{2k-2}$ whenever $(n,N)=1$.
\end{enumerate}
\end{thm}

\begin{proof}
The vector space $M_k(N,\chi)$ is finite-dimensional. Equipped with the Petersson inner product, it admits the orthogonal decomposition
\begin{align*}
  M_k(N,\chi ) = \cE_k(N,\chi ) \oplus S_k(N,\chi ).
\end{align*}
Following this decomposition, we write 
$$
a(n^2)= \sigma(n^2) + \tau(n^2)
$$ 
for each $n\geq1$. If $k$ is an integer, the subspace $\cE_k(N,\chi )$ is spanned by Eisenstein series with Fourier coefficients of the form $\sigma_{k-1}(\chi_1,\chi_2,\delta n)$, where $\chi_1, \chi_2$ are primitive Dirichlet characters modulo $N_1, N_2$ respectively, and $\delta N_1 N_2\mid N$; see \cite[Theorem 8.5.17]{CohenStromberg}. The sum-of-divisors function asymptotic \eqref{eq:divasymp} then gives $$|\sigma(n^2)| \asymp n^{2k-2}$$ when $(n,N)=1$. If $k\in \tfrac{1}{2}+\N$, a basis of  $\cE_k(N,\chi )$ consists of Eisenstein series whose Fourier coefficients evaluated at $n^2$ with $n$ odd are (up to a constant) of the form
$$
\sum_{\delta \mid n} \mu(\delta) \chi'''(\delta) \delta^{k-3/2} \sigma_{2k-2}(\chi',\chi'',n/\delta)
$$
for a quadratic Dirichlet character $\chi'''$ and principal Dirichlet characters $\chi'$ and $\chi''$ of modulus $m$ and $\frac{N}{2m}$, respectively, with $m \vert \frac{N}{4}$; see \cite[Chapter 7]{WangPei2012}. Once again by  multiplicativity, we find that $|\sigma(n^2)| \gg n^{2k-2}$ when $(n,\frac{N}{2})=1$.

We next estimate $\tau (n^2)$. Fix an orthogonal basis of primitive forms $f_1,\dots,f_r$ for $S_k(N,\chi)$ (that may be of lower level --- see \eqref{cusp and new}) and write $\tau(n^2)= \sum_{j=1}^r \scal{f,f_j} a_j(n^2)$, where $a_j(n)$ are the Fourier coefficients at $\infty$ of $f_j$. If $k\in\N$ then Deligne's estimate (the Ramanujan--Petersson conjecture) asserts that $|a_j(n)|\leq d(n)n^{(k-1)/2}$ uniformly for each $n\geq1$. (Here $d(n)$ denotes the standard divisor sum $d(n)=\sigma_0(n)$.) When $k\in \tfrac{1}{2}+\N$, the Ramanujan--Petersson conjecture does not hold for all $n$. However at perfect squares, we have the relation
\begin{align*}
  a_j(n^2) = \sum_{\delta \mid n} \mu(\delta )\chi (\delta )\delta^{k-3/2}A_j(n/\delta ),
\end{align*} 
where $\sum A_j(n)q^n \in S_{2k-1}(N/2,\chi^2)$ is the Shimura lift of $f_j$. Note that here we use that $k\neq\tfrac32$. Because the Shimura lift is Hecke-equivariant, Deligne's estimate again applies, i.e., $|A_j(n)|\leq d(n)n^{k-1}$ for each $n\geq1$.  We conclude with the standard divisor bound that in all cases
$$|\tau(n^2)| \ll_\epsilon  n^{k-1+\epsilon}\sum_{j=1}^r |\scal{f,f_j}|$$ for any $\epsilon>0$. To estimate the inner products (in particular their dependence on $k$), we consider  Parseval's identity $$\scal{f,f}=\sum |\scal{f,f_j}|^2 \scal{f_j,f_j}.$$ On the right hand side, we apply the uniform lower bound $\scal{f_j,f_j}\gg_{\epsilon,N} \tfrac{\G(k)k^{-\eps}}{(4\pi)^k}$ \cite{HoffsteinLockhart1994}. We conclude with an application of the Cauchy--Schwarz inequality and the estimate $r=O_N(k)$.
\end{proof}

\section{Harmonic analysis on the sphere}\label{sec:appendix harmonic analysis}
We refer the reader to the standard reference \cite{SteinWeiss1971}.    

\subsection{Spherical functions}
 Let $d\geq2$. Let $\cH_\nu(d)$ denote the space of all real homogeneous harmonic polynomials of degree $\nu$ in $d+1$ variables. Since any homogeneous polynomial of degree $\nu$ is entirely determined by its values on the unit sphere --- via $P(x)=\|x\|^\nu P(\tfrac{x}{\|x\|})$ --- there is a linear isomorphism between the space $\cH_\nu(d)$ and the space $\cH_\nu(S^{d})$ of spherical harmonics, i.e., of restrictions of such polynomials to the sphere $S^{d}$. We have the dimension formula
$$
n(\nu)\coloneqq \dim \cH_\nu(S^{d}) = \binom{d+\nu}{\nu}-\binom{d-2+\nu}{\nu-2} \asymp \nu^{d-1}
$$
if $d\geq2$ and $\nu\geq2$, while $\dim \cH_1(S^{d})=d$ and $\dim \cH_0 (S^{d})=1$.  

\subsection{Spectral decomposition of $L^2(S^{d})$}
The set of all finite linear combinations in 
$
\bigcup_{\nu\geq0} \cH_\nu(S^{d})
$
is dense in $L^2(S^{d})$, which further admits the direct sum decomposition
$$
L^2(S^{d}) = \bigoplus_{\nu\geq0} \cH_\nu(S^{d}).
$$
In fact each space $\cH_\nu(S^{d})$ is the eigenspace of the Laplacian on $S^{d}$ to the eigenvalue $-\nu(d-1+\nu)$. Further, each function $F\in C^\infty(S^{d})$ has a (unique) spectral expansion of the form
\begin{align}\label{spectral exp}
F = \sum_{\nu\geq0} F_\nu 
\end{align}
with 
\begin{align}\label{eq:def spectral projection}
  F_\nu(x) = \sum_{k=1}^{n(\nu)} \scal{F,Y_\nu^{k}} Y_\nu^{k}(x) = \int_{S^d} F(y) Z_\nu (x,y) d \mu (y),
\end{align}
where $\{Y_\nu^1,\dots, Y_\nu^{n(\nu)}\}$ is an orthonormal basis of $\cH_\nu(S^{d})$, $\scal{\cdot,\cdot}$ is the standard inner product on $L^2(S^d)$ with respect to the normalized rotation-invariant measure $\mu=\tfrac{d \sigma }{\sigma (S^d)}$ on $S^d$,
and 
$$
Z_\nu (x,y) = \sum_{k=1}^{n(\nu )} Y_\nu^k(x) Y_\nu^k(y)
$$
is the zonal harmonic of degree $\nu$. We record the simple properties
\begin{align*}
  \int_{S^d} Z_\nu (x,y) Z_\nu (y,z) d \mu (y) = Z_\nu (x,z), \quad |Z_\nu (x,y)|\leq n(\nu ) \, \text{ for all $x,y,z\in S^d$}.
\end{align*}

\begin{prop}\label{prop:zonal}
  Let $F\in C^{\infty}(S^d)$ and let $F_\nu$ be the spectral projection of $F$ to $\cH_\nu (S^d)$. Then for each $j\in\N_0$ we have
  $\|F_\nu \|_\infty \ll \nu^{d-1-2j}\|\Delta^j F\|_\infty$.
  \end{prop}
  
  \begin{proof}
    From the definition \eqref{eq:def spectral projection} we have the immediate bound
  \begin{align}\label{triv bd}
    |F_\nu (x)| \ll \|F\|_\infty \nu^{d-1}.
  \end{align}
 For the general estimate, we use that 
  $$
  \Delta_x Z_\nu  (x,y) = \Delta_y Z_\nu(x,y) = -\nu(d-1+\nu) Z_\nu(x,y)
  $$ and that $\Delta$ is self-adjoint, so that for each $j\in \N_0$ we have
  \begin{align*}
    F_\nu (x) = \frac{1}{(-\nu(d-1+\nu))^j}\int_{S^d} \Delta^j F(y) Z_\nu(x,y)\, d\mu(y).
  \end{align*}
From this identity the statement follows from $|Z_\nu (x,z)|\leq n(\nu )$ for all $x\in S^d$.
  \end{proof}

\section{Proof of \thmref{Thm2} (and \thmref{Thm2bis})}\label{sec:proofs}

Let $f\in C^\infty(\cE_Q)$ be a test function for the discrete probability measure 
$$
\mu_n(f) \coloneqq \frac{1}{|\Omega_n|} \sum_{x\in\Omega_n} f(x).
$$
By a change of variable, we may address the equidistribution problem using harmonic analysis on the sphere. Let $B$ be the invertible matrix such that $A=2B^T B$ and $Q(x)=\| Bx\|^2$. Then $\cE_Q=B^{-1}(S^d)$. The function $F=f\circ B^{-1}\in C^\infty(S^d)$ can be expressed as the pointwise convergent spectral decomposition $F=\sum F_\nu$, where each $F_\nu$ is the spectral projection of $F$ onto $\cH_\nu (S^d)$. For each $\nu\geq0$, the function $f_\nu  = F_\nu \circ B$ is a spherical function of degree $\nu$ for $A$.  In particular, 
$f_0(x)$ is the constant function $f_0(x)= \int_{\cE_Q} f(y)\, d\mu_Q(y)=\mu_Q(f)$.
With these considerations, we have the spectral expansion
\begin{align}\label{a}
 \sum_{x\in \Omega_n} f(x) =  \mu_Q(f)|\Omega_n|+ \sum_{\nu =1}^\infty \sum_{x\in \Omega_n} f_\nu (x).
\end{align}
For each spherical function $f_\nu$, $\nu>0$, consider the associated theta function 
$$
\Theta = \sum f_\nu ({\bf m})q^{Q({\bf m})} = \sum_{n\geq1} r_\Theta(n) q^n;
$$ 
see \secref{sec:appendix theta}. 
An elementary computation provides the identity
\begin{align*}
\cP_\nu(\Omega_n) \coloneqq   \sum_{x \in \Omega_n} f_\nu (x) = n^{-\nu} \sum_{\delta \mid n}  \mu(\delta )\delta^\nu r_\Theta(n^2/\delta^2).
\end{align*}

\begin{lm}\label{lm:ThetaNorm}
  Let $\Theta$ be the theta function of weight $k$ associated to a positive definite quadratic form $Q$ of level $N$ and a spherical function $P$ of degree $\nu>0$. 
  Then
  \begin{align*}
    \scal{\Theta,\Theta} \ll_{Q,N} \frac{\G(k)}{(4\pi)^{k} } \|P\vert_{\cE_Q}\|^2_{\infty}.
  \end{align*}
  \end{lm}
  
    \begin{proof}
By Rankin--Selberg unfolding (see, e.g., \cite[Chapter 1.6]{BumpBook}) we  have
\begin{align*}
  \int_{\G_0(N)\setminus \h} |\Theta|^2 y^k E(z,s)\, \frac{dxdy}{y^2} &= \frac{\G(s+k-1)}{(4\pi)^{k+s-1}} \sum_{n\geq1} \frac{|r_\Theta(n)|^2}{n^{s+k-1}}
\end{align*}
for $\re(s)>1$, where $E(z,s)$ is the standard nonholomorphic Eisenstein series for $\Gamma_0(N)$ at the cusp at infinity. In particular, the residue of its simple pole at $s=1$ is given by 
    ${\rm vol}(X_0(N))^{-1}$. We may assume that $s>1$ is real. Then from the definition of $r_\Theta(n)$, the right hand-side is smaller than or equal to
    \begin{align*}
       \frac{\G(s+k-1)}{(4\pi)^{k+s-1}} \|P\vert_{\cE_Q}\|_\infty^2 \sum_{n=1}^\infty  \frac{r_Q(n)^2}{n^{s+k-\nu-1}}.
    \end{align*}
    Since $k=\tfrac{d+1}{2}+\nu$, the latter Dirichlet series does not depend on $\nu$. The statement follows by taking $s\to1$ and collecting residues.
    \end{proof}


Applying successively \thmref{thm:bounds on Fourier coefficients}, \lemref{lm:ThetaNorm} and \propref{prop:zonal}, we find that
\begin{align}\label{bound}
  |\cP_\nu(\Omega_n)| 
  &\ll n^{(d-1)/2+\epsilon } \nu^{d-1/2-2j+\epsilon } \|\Delta_A^j f\vert_{\cE_Q}\|_\infty
\end{align}
for each $j\geq0$. Fix a domain $\Psi\subset \cE_Q$ with piecewise smooth boundary. Fix $\eta>0$ and choose $f=\chi_\eta$ to be a smooth approximation of the characteristic function $\chi_\Psi$ such that
\begin{enumerate}
    \item $|\mu_Q(\chi_\eta)-\mu_Q(\Psi)| < \eta $;
    \item $\Delta_A^j \chi_\eta \ll \eta^{-2j}$ for all $j\geq0$ uniformly.
\end{enumerate}
Then upon decomposing the sum over $\nu$ in \eqref{a} into
\begin{align*}
\sum_{\nu \leq \lceil \eta^{-1} \rceil } + \sum_{\nu > \lceil \eta^{-1} \rceil },
\end{align*}
and choosing $2j > d+ \tfrac12$, we find that 
\begin{align*}
\sum_{n\in \Omega_n} \chi_\eta (x) = \mu_Q(\Psi)|\Omega_n| + O(\eta|\Omega_n|) + O(\eta^{-d-1/2}n^{(d-1)/2+\epsilon } ) .
\end{align*}
Averaging over $n=1,\dots,T$ and choosing $\eta = T^{-(d-1)/(2d+3)}$ yields
\begin{align*}
\mu_T(\chi_\eta)= \mu_Q(\Psi)| + O_Q(T^{-(d-1)/(2d+3)+\epsilon}),
\end{align*} 
where we used that $|\Omega_T|\asymp_Q T^d$. Since the right-hand side does not depend on the particular choice of smooth approximation $\chi_\eta$, we approximate the characteristic function supported on $\Psi$ by above and below to reach our conclusion. Similarly, by restricting to $(n,N)=1$,  choosing $\eta = n^{-(d-1)/(2d+3)}$ we find that 
\begin{align*}
  \mu_n(\chi_\eta ) = \mu_Q(\Psi) + O(n^{-(d-1)/(2d+3)+\epsilon}),
\end{align*}
where we use that $|\Omega_n|\asymp_Q n^{d-1}$ as provided by  \thmref{thm:bounds on Fourier coefficients}. This concludes the proof of \thmref{Thm2}.

\begin{rmk}
  The proof of \thmref{Thm2bis} follows along the same lines upon replacing the estimate \eqref{bound} by
  \begin{align}\label{improved bound on average}
 |\cP_\nu(\Omega_T)| \ll T^{d/2+\epsilon } \|f\vert_{\cE_Q}\|_\infty,
  \end{align}
  which will be proven in Section \secref{sec:Proof of Thm1}; see \thmref{thm:sum of P}.
\end{rmk}

\section{Proof of \thmref{thm:modular decomposition}}\label{sec:proof of AC}

  For convenience, we reproduce the statement of \thmref{thm:modular decomposition} here:
  \begin{thm}[\thmref{thm:modular decomposition}]
  Let $k \geq 1$ be an integer or half-integer.  Given a modular form $f\in M_k(N,\chi )$, with $N$ equal to 4 times a squarefree odd integer, and $q$-expansion $f=\sum a(n)q^n$, the associated $L$-series 
    \begin{align*}
       L_f(s)=  \sum_{n=1}^\infty \frac{a(n^2)}{n^s}
    \end{align*}
  converges absolutely for $\re (s)>2k-1$ and admits a meromorphic continuation, which is holomorphic in the half-plane $\re(s)\geq \tfrac{2k-1}{2}$ with the exception of possible poles at $s=2k-1$ and $s=k$.
\end{thm}

The proof proceeds along the same lines as the one for \thmref{thm:bounds on Fourier coefficients}: we first decompose $f$ into a linear combination of Eisenstein series and primitive cusp forms, and then consider each arising $L$-series individually. When the coefficients of the arising $L$-series are Fourier coefficients of a primitive form, the analytic continuation of the $L$-series will follow from the theory of the symmetric square \cite{Shimura1975}, respectively of the Shimura lift \cite{Shimura1973}, depending on the weight $k$. When the coefficients are Fourier coefficients of Eisenstein series, we rely on the explicit identities developed in the next subsection and the explicit computations of \cite[Chapter 7]{WangPei2012}.

\subsection{Identities involving divisor sums}\label{sec:div}
    In this section, let $k$ be a positive integer, $\chi_1$ and $\chi_2$ be Dirichlet characters modulo $N_1$ and $N_2$, respectively. We consider again the divisor sums
    \begin{align*}
      \sigma_k(\chi_1, \chi_2, n)  = \sum_{\delta \mid n} \chi_1(\delta) \chi_2(n/\delta) \delta^k
    \end{align*}
    indexed over all positive divisors $\delta $ of $n$. 
    Since $\sigma_k(\chi_1, \chi_2, n)$ is a Dirichlet convolution, it is multiplicative and we have the identity
    \begin{align}\label{id:1}
    \sum_{n\geq1} \frac{\sigma_k(\chi_1, \chi_2, n)}{n^s} = L(s-k,\chi_1)L(s,\chi_2).
    \end{align}
      Let $\chi=\chi_1 \chi_2$ be the corresponding Dirichlet character modulo $N=N_1 N_2$.     The multiplicativity relation takes the explicit form
    \begin{align}\label{id:mult}
      \sigma_k(\chi_1,\chi_2,m)\sigma_k(\chi_1,\chi_2,n) = \sum_{\delta \mid (m,n)} \chi(\delta ) \delta ^k \sigma_k(\chi_1,\chi_2,mn/\delta ^2)
    \end{align}
valid    for all $m,n\geq1$. 
  
    \begin{lm}\label{lm:Ramanujan}
        When $\re(s)>k+l+1$,  we have the identity
        \begin{align*}
            \sum_{n\geq1} \frac{\sigma_k(\chi_1, \chi_2, n) \sigma_l(\chi_1, \chi_2, n)}{n^s} = \frac{L(s-k-l,\chi_1^2)L(s,\chi_2^2)L(s-k,\chi)L(s-l,\chi)}{L(2s-k-l,\chi^2)}.
            \end{align*}
    \end{lm}
    When $\chi_1$ and $\chi_2$ are trivial, this is a classical identity due to Ramanujan.
    \begin{proof} 
    The Euler product of \eqref{id:1} is given by
    \begin{align*}
    \sum_{n\geq1} \frac{\sigma_k(\chi_1, \chi_2, n)}{n^s} = \prod_p \left( 1 - \alpha_{1,p} p^{-s} \right)^{-1} \left( 1 - \alpha_{2,p} p^{-s} \right)^{-1},
    \end{align*}
    where the complex numbers $\alpha_{1,p}$ and $\alpha_{1,p}$ satisfy the relations
    \begin{align*}
    \alpha_{1,p} + \alpha_{2,p} &=  \chi_1(p) p^k+\chi_2(p) \\
    \alpha_{1,p} \, \alpha_{2,p} &= \chi_1(p)\chi_2(p) p^k = \chi(p)p^k. 
    \end{align*}
    The Rankin--Selberg convolution of \eqref{id:1} then has Euler product
    \begin{align*}
      L(2s-k-l,\chi^2) \sum_{n\geq1} \frac{\sigma_k(\chi_1, \chi_2, n) \sigma_l(\chi_1, \chi_2, n)}{n^s} = \prod_p \prod_{i,j} \left( 1 - \alpha_{i,p}\beta_{j,p} p^{-s} \right)^{-1} 
    \end{align*}
    (see, e.g., \cite[Theorem 1.63]{BumpBook}), which, given the relations above, is equal to 
    \begin{align*}
    L(s-k-l,\chi_1^2)L(s,\chi_2^2)L(s-k,\chi  )L(s-l,\chi).
    \end{align*}
    The statement follows.
    \end{proof}

    \begin{lm}\label{lm:2}
    When $\re(s)>2k+1$, we have
    \begin{align}\label{id:n^2}
      \sum_{n\geq1} \frac{\sigma_k(\chi_1,\chi_2,n^2)}{n^{s}} 
      &=  \frac{L(s-2k,\chi_1^2)L(s,\chi_2^2)L(s-k,\chi)}{L(2s-2k,\chi^2) }. 
    \end{align}
    \end{lm}
    \begin{proof}
    By the multiplicativity relation \eqref{id:mult} we have
    \begin{align*}
      \sum_{n\geq1} \frac{\sigma_k(\chi_1,\chi_2,n)^2}{n^{s}} 
      &= L(s-k,\chi) \sum_{n\geq1} \frac{\sigma_k(\chi_1,\chi _2,n^2)}{n^s} ,
    \end{align*}
    and the statement follows by applying \lemref{lm:Ramanujan} to the left hand side.
    \end{proof}
    \begin{lm}\label{lm:3}
       Let $\delta $ be a positive integer of the form $\delta =2^j \delta '$ with $\delta '$ odd and squarefree. Then for $\re(s)$ sufficiently large we have
       \begin{align*}
        \sum_{\substack{n=1\\ \delta\mid n^2}}^\infty \frac{\sigma_k(\chi_1,\chi_2,n^2/\delta )}{n^s} &=
        \frac{(2^{\lceil j/2 \rceil}\delta ' )^{-s} \sigma_k(\chi _1,\chi _2,2^\epsilon \delta ')}{\sigma_{k-s}(\chi ,2^\epsilon \delta ') }\sum_{n=1}^\infty \frac{\sigma_k(\chi _1,\chi _2,n^2)}{n^s},
       \end{align*}
        where $\sigma_t(\chi,n)=\sum_{\delta \mid n} \chi (\delta )\delta ^t$, and $\epsilon=0$ if $j$ is even while $\epsilon=1$ if $j$ is odd.
    \end{lm}
   
    \begin{proof}
      The left hand side is equal to
    \begin{align*}
        & \sum_{\substack{n=1\\2^{\lceil j/2 \rceil} \delta ' \mid n}}^\infty \frac{\sigma_k(\chi _1,\chi _2,n^2/\delta )}{n^s}= (2^{\lceil j/2 \rceil}\delta ' )^{-s} \sum_{n=1}^\infty \frac{\sigma_k(\chi _1,\chi _2,2^\epsilon \delta ' n^2)}{n^s},
    \end{align*}
    where $\epsilon=0$ if $j$ is even and $\epsilon=1$ if $j$ is odd. Via  \eqref{id:mult} we take note of 
    \begin{align*}
      \sigma_k(\chi _1,\chi _2,2^\epsilon \delta ')\sigma_k(\chi _1,\chi _2,n^2) = \sum_{\mu \mid (2^\epsilon \delta ',n)}\chi(\mu )\mu ^k \sigma_k(\chi _1,\chi _2,2^\epsilon \delta 'n^2/\mu ^2).
    \end{align*}
With this multiplicativity relation we deduce that
    \begin{align*}
      \sigma_k(\chi _1,\chi _2,2^\epsilon \delta ')\sum_{n=1}^\infty \frac{\sigma_k(\chi _1,\chi _2,n^2)}{n^s} = \sigma_{k-s}(\chi, 2^\epsilon \delta ')  \sum_{n=1}^\infty \frac{\sigma_k(\chi _1,\chi _2,2^\epsilon \delta ' n^2)}{n^s} .
    \end{align*}
    The statement follows. 
    \end{proof}
    
    \begin{rmk}\label{rmk:silly}    
    Because the sum-of-divisors function is multiplicative we have the elementary identity
    \begin{align*}
        \sigma_{k-s}(\chi,2^\epsilon \delta') =\prod_{p\mid 2^\epsilon \delta '} (1+\chi(p)p^{k-s}).
    \end{align*}
From this we observe that $\sigma_{k-s}(\chi,2^\epsilon \delta')=0$ implies that $s=k$.
    \end{rmk}

\begin{lm}\label{lm:odd}
Let $k,j\in\Z$. Then
\begin{align*}
\sum_{\substack{n\geq1\\n\text{ odd}}} n^{-s} & \sum_{\delta\mid n} \mu(\delta)  \chi_3(\delta) \delta^{j} \gs_{k}(\chi_1, \chi_2, n/\delta) \\
&= \frac{(1 - \chi_2(2)2^{-s} )(1- \chi_1(2) 2^{k-s} )}{1-\chi_3(2)2^{j-s}  } \frac{L(s-k,\chi_1) L(s, \chi_2)}{L(s-j,\chi_3)}
\end{align*}
when $\re(s)>\max\{j,k\}.$
\end{lm}

\begin{proof}
We first observe that without the restriction to odd integers we have
\begin{align*}
\sum_{n\geq1}n^{-s} \sum_{\delta\mid n} \mu(\delta) \chi_3(\delta) \delta^{j} \gs_{k}(\chi_1, \chi_2, n/\delta) &=
\frac{L(s-k,\chi_1) L(s, \chi_2)}{L(s-j,\chi_3)}.
\end{align*}
The contribution to this sum of the even integers only is
\begin{align*}
& \sum_{\delta\text{ odd}} \mu(\delta) \chi_3(\delta) \delta^{j} \sum_{2d \mid n}\frac{\gs_{k}(\chi_1, \chi_2, n/\delta)}{n^s} +\sum_{\delta\text{ even}} \mu(\delta) \chi_3(\delta) \delta^{j} \sum_{d \mid n}\frac{\gs_{k}(\chi_1, \chi_2, n/\delta)}{n^s} \\
&= \left( \frac{\gs_{k}(\chi_1, \chi_2, 2) - \chi(2) 2^{k-s} }{2^s} \sum_{\delta\text{ odd}} \frac{\chi_3(\delta)\mu(\delta)}{\delta^{s-j}} +\sum_{\delta\text{ even}} \frac{\chi_3(\delta)\mu(\delta)}{\delta^{s-j}} \right) L(s-k,\chi_1) L(s, \chi_2).
\end{align*}
Noting that
\begin{align*}
\sum_{\delta\text{ odd}}\frac{\chi_3(\delta)\mu(\delta)}{\delta^{s-j}} = \prod_{p>2} (1-\chi_3(p)p^{j-s}) = \op{(1-\chi_3(2)2^{j-s}) L(s-j,\chi_3)}^{-1},
\end{align*}
and combining the above expressions yields the identity.
\end{proof}

\subsection{Proof of \thmref{thm:modular decomposition}}\label{sec:proof of thm} Fix a basis $\{f_1,\dots f_r\}$ for $M_k(N,\chi)$. We consider the cases of integer and half-integer weight $k$ separately.

\vspace{.2cm}
\noindent {\bf Case 1: $k\in \N$}. If $f_j= \sum a_j(n)q^n$ is not a cusp form and $k\geq 3$, we may assume that $f_j(z)$ is the Eisenstein series $G_k(\chi_1,\chi_2,\delta  z)$ (see \cite[Theorem 8.5.17]{CohenStromberg}), with $\chi_1, \chi_2$ primitive characters modulo $N_1, N_2$ and $\delta $ a positive integer such that $\delta N_1N_2\mid N$; see \cite[Theorem 8.5.17]{CohenStromberg}. The Fourier coefficients of $G_k(\chi_1,\chi_2,z)$ are given (up to a constant multiple) by the divisor sums $\sigma_{k-1}(\chi_1, \chi_2,n)$. The associated $L$-series is therefore  given by
\begin{align*}
   \sum_{\substack{n=1\\ \delta \mid n^2} }^\infty \frac{\sigma_{k-1}(\chi_1,\chi_2,n^2/\delta )}{n^s}
\end{align*}
(again up to a constant multiple). Under the assumption that $N$ is 
4 times a squarefree integer, we apply successively \lemref{lm:3} and \lemref{lm:2} to find that for $\re(s)>2k-1$ this Dirichlet series is equal to 
\begin{align}\label{A}
    \frac{L(s-2k+2,\chi_1^2)L(s,\chi_2^2)L(s-k+1,\chi)}{L(2s-2k+2,\chi^2) }
\end{align}
times a function that is holomorphic when $\re(s)>k-1$; see here \rmkref{rmk:silly}. We can conclude that \eqref{A} is absolutely convergent when $\re(s)>2k-1$ with a meromorphic continuation to the half-plane $\re(s)\geq k-\tfrac{1}{2}$ that is holomorphic except for possible simple poles 
\begin{itemize}
\item at $s=2k-1$ if $\chi_1^2$ is principal;
\item at $s=k$ if $\chi$ is principal.
\end{itemize}
The cases $k=1,2$ yield the same result by building on the explicit bases given in \cite[Theorems 8.5.2, 8.5.3]{CohenStromberg}.

When $f_j$ is a cusp form, we may assume that $f_j$ is primitive; see \cite[Corollary 13.3.6]{CohenStromberg} and discard multiplicative factors that are holomorphic functions in the half-plane $\re(s)>k-1$ just as we did above. The $L_j(s)$ can be expressed in terms of the symmetric square $L$-function as
\begin{align}\label{eq:cusp-sym square}
   L_j(s)=  \sum_{n=1}^\infty \frac{a_j(n^2)}{n^s} = \frac{L(s-k+1,\sym^2 f_j)}{L(2s-2k+2,\chi^2)}
\end{align}
(see \secref{sec:appendix modular forms and L-functions}).
It follows that $L_j(s)$ admits a holomorphic continuation to the half-plane $\re(s)\geq k-\tfrac{1}{2}$ except for a possible simple pole at $s=k$. 

\vspace{.2cm}

\noindent {\bf Case 2: $k\in \tfrac{1}{2}+\N$.} We first suppose that  $f_j$ is a cusp form. Then $f_j$ has a Shimura lift $\tilde f_j= \sum A_j(n)q^n \in M_{2k-1}(N/2,\chi^2)$, where the Fourier coefficients $A_j(n)$ are completely determined by the relation 
\begin{align}\label{id:Di}
   L_j(s) = \sum_{n=1}^\infty \frac{a_j(n^2)}{n^s} = \frac{1}{L(s-k+\tfrac{3}{2},\tilde\chi)}\sum_{n=1}^\infty \frac{A_j(n)}{n^s}
\end{align}
 (see \secref{sec:appendix Shimura}). If $k \geq \tfrac{5}{2}$, then $\tilde f_j$ is a cusp form, and consequently the Dirichlet series \eqref{id:Di} defines a holomorphic function in the half-plane $\re(s)\geq k-\tfrac{1}{2}$. If $k=\tfrac{3}{2}$, then $\tilde f_j$ may not be cuspidal, and therefore \eqref{id:Di} is holomorphic in the half-plane $\re(s)\geq k-\tfrac{1}{2}$ except for a possible simple pole at $s=k$.

If $f_j$ is not cuspidal, we need to directly study half-integral weight Eisenstein series\footnote{If $k-\tfrac{1}{2}$ is even, we could once more argue via the Shimura lift $\tilde f_j \in M_{2k-1}(N/2,\chi^2)$, whose non-zero Fourier coefficients are determined by the relation \eqref{id:Di}.}. A basis of $\cE_k(N,\chi )$ is provided by forms whose Fourier coefficients are given in terms of special values of Dirichlet $L$-functions and modified divisors sums \cite[Chapter 7]{WangPei2012}. We find that the Dirichlet series $L_j(s)$ for each basis element $f_j \in  \cE_k(N,\chi )$ is (up to a constant multiple) of the form
\begin{align*}
\op{1+\frac{1}{1-2^{2k-2-s}}}
\sum_{\substack{f\geq1\\ f\text{ odd}}} f^{-s}\sum_{\delta\mid f} \mu(\gd) \chi'''(\gd) \gd^{k-\frac32} \gs_{2k-2}(\chi', \chi'', f/\gd),
\end{align*}
where $ \chi', \chi''$ are principal Dirichlet character (with $\chi'(2)=1$) and $\chi'''$ a quadratic Dirichlet character. By \lemref{lm:odd} this is equal to
\begin{align*}
\frac{ L(s-2k+2, \chi')L(s,\chi'')}{L(s-k+3/2, \chi''' )}
\end{align*}
times a function that is holomorphic when $\re(s)>k-\frac32$. Hence the Dirichlet series $L_j(s)$ can be identified with a holomorphic function in the half-plane $\re(s)\geq k-\tfrac{1}{2}$ except for a simple pole at $s=2k-1$. This concludes the proof of \thmref{thm:modular decomposition}.

\subsection{Specializing \thmref{thm:modular decomposition} to the generating series $R_\Theta(s)$}\label{sec:Collecting/Specializing}
Under the assumptions made on $A$ at the beginning of this paper, the theta function 
\begin{align*}
  \Theta  = \sum_{{\bf m}\in \Z^{d+1}} P({\bf m})q^{Q({\bf m})} = \sum_{n\geq0} r_\Theta (n) q^n
\end{align*} 
is a holomorphic modular form of weight $k=\tfrac{d+1}{2}+\nu$, level $N$, and Nebentypus
\begin{align*}
  \chi = \begin{dcases}
(\tfrac{(-1)^{(d+1)/2}\det(A)}{\cdot}), &\text{ if } d \text{ is odd} ;\\
(\tfrac{2\det(A)}{\cdot}), &\text{ if } d \text{ is even}.
\end{dcases}
\end{align*}
Set $\cP(\Omega_n)= \sum_{x\in \Omega_n} P(x)$ and 
\begin{align*}
    R_\Theta (s)= \sum_{n=1}^\infty \frac{\cP(\Omega_n)}{n^s} = \frac{1}{\zeta (s)} \sum_{n=1}^\infty \frac{n^{-\nu }r_{\Theta }(n^2)}{n^s}.
\end{align*} 
We assume throughout that $d\geq2$.
\subsection*{Case $\nu > 0$:} In this case $\Theta \in S_k(N,\chi)$. The proof of \thmref{thm:modular decomposition} gives the identities
\begin{align*}
  R_\Theta (s) = \begin{dcases}
      \frac{L(s-\tfrac{d-1}{2},\tilde \Theta )}{\zeta (s)L(s-\tfrac{d}{2}+1,\tilde\chi ) }, &\text{ if } d \text{ is even};\\
      \frac{L(s-\tfrac{d-1}{2},\sym^2 \Theta)}{\zeta (s)L(2s-d+1,\chi)}, &\text{ if } d \text{ is odd} ;
  \end{dcases}
\end{align*} 
where $\tilde \Theta\in S_{d+2\nu}(N/2,\chi ^2)$ is the 
Shimura lift of $\Theta$.

Hence $R_\Theta(s)$ has a meromorphic continuation, which is holomorphic in $\re(s)\geq \tfrac{d}{2}$ except for a possible pole at $s=\tfrac{d+1}{2}$ if $d$ is odd and $\chi$ is nonprincipal; see \cite[Theorem 2]{Shimura1975}.

\subsection*{Case $\nu =0$ and $d$ odd:} In this case, $P=1$, $\cP(\Omega_n)=|\Omega_n|$ and $R_\Theta=R_Q$. Here, $R_\Theta(s)$ is a linear combination of summands of the form 
\begin{align*}
  \frac{\zeta (s-d+1)L (s+\tfrac{1-d}{2},\chi)}{L(2s-d+1,\chi^2)}
\end{align*}
for some quadratic character $\chi$ (times a function that is holomorphic when $\re(s)>\tfrac{d-1}{2}$ and uniformly bounded in $t$),  and functions that are holomorphic in $\re(s)\geq \tfrac{d}{2}$ except for a possible pole at $s=\tfrac{d+1}{2}$ (the contribution of the cusp forms). Here we used that each appearing character $\chi$ is quadratic and hence we have $L(s,\chi^2)=\zeta_N(s)\zeta (s)$, where $\zeta_N(s)=\prod_{p\mid N}(1-p^{-s})^{-1}$ is the local zeta-function. Hence the meromorphic continuation of $R_\Theta(s)$ to $\re(s)\geq \tfrac{d}{2}$ has a simple pole at $s=d$ and possibly another one at $s=\tfrac{d+1}{2}$. The pole at $s=\tfrac{d+1}{2}$ is actually removable if $d\equiv 3$ (mod 4) and $\chi$ is principal. Indeed, since the Riemann zeta function has a zero at each even negative integer, the residue is 
\begin{align*}
    \lim_{s\to (d+1)/2)} (s-\tfrac{d+1}{2})\frac{\zeta (s-d+1)\zeta (s+\tfrac{1-d}{2})}{\zeta (2s-d+1)} = \frac{\zeta (\tfrac{3-d}{2})}{\zeta (2)}=0.
\end{align*}

\subsection*{Case $\nu =0$ and $d$ even:} Then $R_\Theta (s)$ is a linear combination of summands of the form
\begin{align*}
 \frac{\zeta(s-d+1)}{L(s-\frac{d}{2}+1, \chi' )} 
\end{align*}
for a Dirichlet character $\chi'$ (times a function that is holomorphic when $\re(s)>\tfrac{d-1}{2}$ and uniformly bounded in $t$), and functions that are holomorphic when $\re(s)\geq\tfrac{d}{2}$ (the contribution of the cusp forms).  Hence $R_\Theta(s)$ is holomorphic in $\re(s)\geq \tfrac{d}{2}$ except for a simple pole at $s=d$.

\section{Proof of \thmref{thm:sym square mean square}}\label{sym square mean square}

\subsection{Preliminary lemmata}
Let $f\in S_k(N,\chi)$ be a primitive newform with $q$-expansion $f=\sum a(n)q^n$ and define $(b(n))_{n\geq1}$ by
$
L(s,\sym^2 f) = \sum_{n\geq1} b(n)n^{-s} 
$ 
(for $\re(s)>1$).

\begin{lm}
For any $l\geq0$, we have
$$
\sum_{n\leq N} |b(n)|^l \ll  N(\log N)^{2^{2l}-1}.
$$
\end{lm}
\begin{proof}
By construction, we have
$$
b(n)= n^{1-k}\sum_{m^2\mid n} \chi^2(m)a(n^2/m^4)m^{2k-2}
$$
for every $n\geq1$. Deligne's bound asserts that $|a(n)| \leq d(n)n^{\tfrac{k-1}{2}}$ uniformly and the multiplicativity of the divisor function yield
$$
|b(n)| \ll \sum_{m\mid n} d(m^2) = d(n)^2.
$$
The statement now follows from the following inequality of Ramanujan
$$
\sum_{n\leq x} d(n)^l \ll x(\log x)^{2^l-1};
$$
see \cite{Wilson1923}.
\end{proof}

We also recall the Montgomery--Vaughan mean value theorem for Dirichlet polynomials \cite[Corollary 3]{MontgomeryVaughan1974}.

\begin{lm} \label{MVlemma} For complex numbers $a_1,\dots, a_N\in\C$, we have
\begin{align}\label{MV}
\int_0^T \left| \sum_{n\leq N} \frac{a_n}{n^{it}}\right|^2\, dt = \sum_{n\leq N} |a_n|^2\op{T+O(n)}.
\end{align}
\end{lm}

\subsection{Approximate functional equation}

Let $f\in S_k(N,\chi )$ be a primitive form. Its symmetric square $L$-function and Rankin--Selberg convolution $L$-function are related by
\begin{align*}
  L(s,\chi )L(s,\sym^2 f)= L(s,f\otimes f).
\end{align*}
The precise functional equation for the symmetric square $L$-function, 
\begin{align}\label{eq:FE}
  L (s,\sym^2 f) = \Xi (s) L(1-s,\sym^2 f),
\end{align} 
can be explicily derived from \cite{Li1979}; we will only take note of the fact that $\Xi(s)$ takes the form
\begin{align*}
  \Xi (s) = A(s) \frac{\Gamma (k-s)\Gamma (\tfrac{2-\lambda -s}{2})}{\Gamma (s+k-1)\Gamma (\tfrac{s+1-\lambda }{2})},
\end{align*}
where $\lambda$ is determined by $\chi(-1)=(-1)^\lambda$ and $|A(s)|=O(1)$ uniformly for $a\leq \sigma \leq  b$ \cite{Li1979}.

\begin{thm}
Let $Y\ll T^c$ (for some $c>0$), $M\geq Y^{-1}T^3$. Set $h=\log^2 T$. Then for $t\in[h^2,T]$ we have
\begin{align*}
L(\gs&+it,\sym^2 f) = \sum_{n\leq 2Y} \frac{b(n)}{n^{\gs+it}} e^{-(n/Y)^h}\\
&\,  - \frac{1}{2\pi i}  \int_{\substack{\re(w)=-1\\ |\im(w)|\leq h^2}} \Xi(\gs+it+w) \sum_{n\leq M} \frac{b(n)}{n^{1-\gs-it-w}}Y^w \G(1+\tfrac{w}{h})\frac{dw}{w} + o(1)
\end{align*}
as $T\to\infty$ uniformly in $0\leq \gs\leq1$.
\end{thm}
\begin{proof}
Let $Y>0$, $h>2$. By Mellin inversion, we have
\begin{align*}
\sum_{n\geq1} \frac{b(n)}{n^s} e^{-(n/Y)^h} &= \frac{1}{2\pi i} \int_{(c/h)} L(s+hz,\sym^2 f) Y^{hz}\G(z)\, dz\\
& =  \frac{1}{2\pi i} \int_{(c)}  L(s+w,\sym^2 f) Y^w \G(1+\tfrac{w}{h})\frac{dw}{w}
\end{align*}
for $c>1/2$. We shift the contour of integration to $\re(w)=-1$ and find
\begin{align*}
\sum_{n\geq1} \frac{b(n)}{n^s} e^{-(n/Y)^h} &= L(s,\sym^2 f) + \frac{1}{h}\G(\tfrac{\overline s}{h})Y^{\overline s} {\rm Res}(L(s,\sym^2 f), s=1)\\
&\qquad + \frac{1}{2\pi i} \int_{(-1)}  L(s+w,\sym^2 f) Y^w \G(1+\tfrac{w}{h})\frac{dw}{w}.
\end{align*}
We recall Stirling's asymptotic formula
$$
|\G(\gs+i\tau)| \sim \sqrt{2\pi} |\tau|^{\gs-1/2} e^{-\pi|\tau|/2}
$$
as $|\tau|\to\infty$ uniformly for all $a\leq \gs\leq b$. Write $s=\gs+it$, with $0\leq\gs\leq1$. We have
 $$
h^{-1}\left|\G(\tfrac{\overline s}{h}){\rm Res}(L(s,\sym^2 f), s=1)\right|Y^{1/2} \ll |t|^{\gs/h-1/2} e^{-\pi|t|/2h} Y^{1/2}
 $$
as $T\to\infty$. Restricting $t$ to the interval $[h^2,T]$ and choosing $Y=T^c$ (for some $c>0$), $h=\log^2 T$, this is $o(1)$.  These choices also yield
$$
\sum_{n>2Y} \frac{b(n)}{n^s} e^{-(n/Y)^h} = o(1).
$$
Let $M\geq1$. By the functional equation \eqref{eq:FE} we have
\begin{align*}
 \int_{(-1)} & L(s+w,\sym^2 f) Y^w \G(1+\tfrac{w}{h})\frac{dw}{w} \\
&=  \int_{(-1)} \Xi(s+w) L(1-s-w,\sym^2 f) Y^w \G(1+\tfrac{w}{h})\frac{dw}{w} = I_1 + I_2,\\
  \end{align*}
  where
\begin{align*}
I_1& \coloneqq  \int_{(-1)} \Xi(s+w) \sum_{n\leq M} \frac{b(n)}{n^{1-s-w}}Y^w \G(1+\tfrac{w}{h})\frac{dw}{w}\\
& =  \int_{\substack{\re(w)=-1\\ |\im(w)|\leq h^2}} \Xi(s+w) \sum_{n\leq M} \frac{b(n)}{n^{1-s-w}}Y^w \G(1+\tfrac{w}{h})\frac{dw}{w} + o(1)\\
\end{align*}
and, by shifting contour,
\begin{align*}
I_2 & =   \int_{\re(s+w)=-h/2} \Xi(s+w)  \sum_{n>M} \frac{b(n)}{n^{1-s-w}} Y^w \G(1+\tfrac{w}{h})\frac{dw}{w}.
\end{align*}
To estimate the latter integral we use that
$$
\left|\sum_{n>M} \frac{b(n)}{n^{1-s-w}}\right| \ll hM^{-h/2+\eps}
$$
and
$$
|\Xi(-\tfrac{h}{2}+i\tau)| \ll |\tau|^{3/2(h+1)}
$$
for $|\tau|\gg1$. Then
$$
|I_2| \ll hM^{\eps-h/2}Y^{-h/2-\gs}\op{T+h^{-1}\int_T^\infty \tau^{3/2(h+1)}(\tau/h)^{-\gs/h} e^{-\pi \tau/2h} d\tau},
$$
which is $o(1)$ as $T\to\infty$ if we require that $M\geq Y^{-1}T^3$. 
\end{proof}

\subsection{Proof of the integral mean square estimate}\label{sec:mean square}

Following the approximate functional equation, the main term contribution is given by an application of the Montgomery--Vaughan inequality of \lemref{MVlemma}:
\begin{align*}
\int_0^T \ov{\sum_{n\leq 2Y} b(n)n^{-1/2-it} e^{-(n/Y)^h}}^2 dt &\ll T\sum_{n\leq 2Y} b(n)^2 n^{-1}+ \sum_{n\leq2Y} b(n)^2\\
&\ll T(\log Y)^{16}+Y(\log Y)^{15}.
\end{align*}
We now turn to the integral $I_1$ in the approximate functional equation, shift the contour of integration to $\re(w)=-1/2$ and apply the functional equation to the integrand. Writing $s=1/2+it$, $w=-1/2+iv$, this yields, for $t\asymp T$,
\begin{align*}
\int_{-\infty}^\infty & \ov{\frac{v}{h}}^{1/2-1/(2h)} e^{-\pi/2|v/h|} |t+v|^{3/2} \ov{\sum_{n\leq M} \frac{b(n)}{n^{1-i(t+v)}}}dv \\
&\quad \ll \int_{|v|\leq h} (\dots) +o(1) \ll T^{3/2} \int_{|v|\leq h} \ov{\sum_{n\leq M} \frac{b(n)}{n^{1-i(t+v)}}} dv +o(1).
\end{align*}
Applying Cauchy's inequality and the Montgomery--Vaughan inequality once more, we get
\begin{align*}
\int_0^T |I_1|^2 dt &\ll Y^{-1} T^3  \int_{|v|\leq h} \int_0^T \ov{ \sum_{n\leq M} b(n) n^{-1+iv} n^{it}}^2 dt\, dv\\
&\ll  Y^{-1} T^3 (\log T)^2 \op{T\sum_{n\leq M} b(n)^2 n^{-2} + \sum_{n\leq M} b(n)^2 n^{-1}}\\
&\ll Y^{-1} T^3  \op{TM^{-1}(\log M)^{15} +(\log M)^{16}}.
\end{align*}
Choosing $M=Y=T^{3/2}$ we conclude that
$$
\int_0^T \ov{L(\tfrac12+it,\sym^2 f)}^2 dt\ll T^{3/2}(\log T)^{18}.
$$

\section{Proof of \thmref{Thm1}}\label{sec:Proof of Thm1}

The first step in the proof of \thmref{Thm1} is the following truncation of Perron's formula.

\begin{prop}\label{prop:Perron}
Let $P$ be a spherical function. Fix $\epsilon >0$, $T\not\in\N$ and choose $\beta > \frac{d+1}{2}$  such that it is greater than the abscissa of absolute convergence of $R_\Theta (s)$. Then
  \begin{align*}
\cP(\Omega_T) \coloneqq \sum_{x\in \Omega_T} P(x) =   \frac{1}{2\pi i}  \int_{\beta -iH}^{\beta +iH} R_\Theta (s) \frac{T^s}{s}\, ds + O\op{\frac{T^\beta }{H}}.
  \end{align*}
\end{prop}
\begin{proof}
We may write 
\begin{align*}
  \frac{1}{2\pi i}\int_{\beta -iH}^{\beta +iH} R_\Theta (s) \frac{T^s}{s}ds = \sum_{n=1}^\infty  \frac{\cP(\Omega_n)}{2\pi i}\int_{\beta -iH}^{\beta +iH} \frac{T^s}{n^s} \frac{ds}{s}.
\end{align*}
The computation of the integral is a standard exercise in complex analysis and yields
\begin{align*}
\sum_{n\leq T} \cP(\Omega_n) + O\op{\frac{T^\beta }{H} \sum_{n=1}^\infty \frac{|\cP(\Omega_n)|}{n^\beta}}.
\end{align*}
We know from \eqref{bound}  that $|\cP(\Omega_n)|\ll n^{\frac{d-1}{2}}$ for every $n\geq 1$ and hence the sum on the right converges absolutely.  
\end{proof}

We will compute this integral via Cauchy's residue theorem applied to the rectangle $\cR$ with side edges $[\beta - iH,\beta +iH]$ and $[d/2+iH,\beta + iH]$. For this we will rely on asymptotics and bounds on the integral mean square of various $L$-functions. We recall the classical mean square asymptotic
\begin{align} \label{asym:zeta}
\int_{0}^T |\gz(\gs+it)|^2 dt \sim (2\pi)^{2\gs-1} \frac{\gz(2-2\gs)}{2-2\gs}\, T^{2-2\gs} 
\end{align}
as $ T \rightarrow \infty$, when $\sigma<\tfrac12$. We will also need

\begin{prop}\label{prop:ms-recipro} 
  Let $\chi$ be a nonprincipal Dirichlet character mod $q$ and let $\chi_0$ be the principal Dirichlet character mod $q$. As $T\to\infty$, we have
 \begin{align}\label{zerofree}
 \int_0^T \frac{dt}{|L(\sigma +it,\chi)|^2} \sim \frac{L(2\sigma ,\chi_0)}{L(4\sigma ,\chi_0^2)} T
 \end{align}
 for $\sigma=1$ and the asymptotic remains true for all $\sigma>\tfrac{1}{2}$ if we assume GRH.
 \end{prop}
 \begin{proof}
 A standard mean-value theorem for Dirichlet series (see, e.g., \cite[Theorem 7.1]{Titchmarsh1986}) implies that
 \begin{align}\label{eq:Tit}
   \int_0^T \left|\frac{1}{L(\sigma +it,\chi )}\right|^2\, dt \sim \sum_{n=1}^\infty \frac{|\mu(n)|\chi_0(n)}{n^{2\sigma }}\, T = \frac{L(2\sigma,\chi_0 )}{L(4\sigma ,\chi _0^2)} T
 \end{align}
 whenever $\sigma>1$. If $\chi$ is nonprincipal,  $L(s,\chi)^{-1}$ has an analytic continuation to $\re(s)=1$ by the prime number theorem in arithmetic progressions. By the Ingham--Newman Tauberian theorem, the Dirichlet series $\sum \mu (n)\chi (n) n^{-s}$ converges to $L(s,\chi)^{-1}$ for $\Re(s)\geq1$. We conclude that \eqref{eq:Tit} also holds for $\sigma =1$. If we moreover assume that GRH is true, then
 $$
 \int_0^T \frac{dt}{|L(\sigma+it,\chi)|^2} \sim \frac{L(2\sigma,\chi_0)}{L(4\sigma,\chi_0^2)}T
 $$
 for every fixed $\sigma>\tfrac12$; see \cite[p.338]{Titchmarsh1986}. 
 \end{proof} 

The same arguments go to show that
\begin{align}\label{asym:L}
\int_{0}^T \left|\frac{L(\gs+it,\chi)}{L(2\gs+2it,\chi^2)}\right|^2 dt \sim \frac{L(2\sigma ,\chi_0)}{L(4\sigma ,\chi_0^2)} T
\end{align}
as $T\to\infty$, for $\sigma \geq 1/2$. We can now prove \thmref{Thm1}.
\begin{proof}[Proof of \thmref{Thm1}]
Consider the theta function $\Theta$ associated to $Q$; we recall that $\Theta \in M_{(d+1)/2}(N,\chi)$. Fix a basis $\cB$ of $M_{(d+1)/2}(N,\chi)$ and consider the linear decomposition $R_Q (s) = \sum_{f\in \cB} \scal{\Theta ,f} R_f(s)$. Perron's formula leads us to consider each integral
\begin{align*}
  \frac{1}{2\pi i} \int_{(\beta )} R_f(s) \frac{T^s}{s}\, ds = \frac{1}{2\pi i}\int_{\beta -iH}^{\beta +iH} R_f(s) \frac{T^s}{s}\, ds + O\op{\frac{T^\beta }{H}}
\end{align*}
separately. From the results of \secref{sec:Collecting/Specializing}, we know that each $R_f(s)$ is of the form 
\begin{table}[H]
  \centering
  \begin{tabular}{c|c|c}
      \toprule
       & $f$ not cuspidal  & $f$ cuspidal   \\
    \midrule
     $d$ even  & $\frac{\zeta (s-d+1)}{L(s-\tfrac{d}{2}+1,\chi')}$  & $\frac{L(s-\tfrac{d-1}{2}, \tilde f)}{\zeta (s)L(s-\tfrac{d}{2}+1,\tilde{\chi})}$  \\
     $d$ odd  & $\frac{\zeta (s-d+1)L(s+\tfrac{1-d}{2},\chi)}{L (2s-d+1,\chi^2)}$ & $\frac{L(s-\tfrac{d-1}{2}, \sym^2 f)}{\zeta (s)L(2s-d+1,\chi )}$  \\
      \bottomrule
  \end{tabular}
\end{table}
\noindent
times a function in $s$ that is holomorphic for $\re(s)>\tfrac{d-1}{2}$ and uniformly bounded for $\re(s)\geq \beta$. In particular each entry determines a Dirichlet series that is absolutely convergent when $\re(s)>d$ and admits a holomorphic continuation to $\re(s)\geq \tfrac{d}{2}$, except for 
\begin{itemize}
\item a simple pole at $s=d$ when $f$ is not cuspidal;
\item a possible simple pole at $s=\tfrac{d+1}{2}$ when $\chi$ is principal and $d\equiv 1$ (mod 4) when $f$ is not cuspidal;
\item a possible simple pole at $s=\tfrac{d+1}{2}$ when $\chi$ is nonprincipal and $d$ is odd when $f$ is cuspidal. 
\end{itemize}
For each standard 
$L$-function $L(s,\pi)$ of degree $D$ we have the convexity bound
\begin{align*}
    |L(\sigma +it,\pi)| \ll \begin{dcases}
        1, &\text{ if }\sigma >1  ;\\
       t^{\tfrac{D}{2}(1-\sigma)} , &\text{ if } 0\leq \sigma\leq 1 ;\\
        |t|^{D(\tfrac{1}{2}-\sigma)}, &\text{ otherwise} .
    \end{dcases}
\end{align*}
Together with the standard estimate $|L(s,\chi)|^{-1}\ll_\epsilon |t|^\epsilon$ for $\sigma= 1$ and $|L(s,\chi)|^{-1}=O(1)$ for $\sigma >1$ we find that $|R_f(\tfrac{d}{2}+it)|$ is
\begin{table}[H]
  \centering
  \begin{tabular}{c|c|c}
      \toprule
       & $f$ not cuspidal  & $f$ cuspidal   \\
    \midrule
     $d$ even  & $O_\epsilon (|t|^{(d-1)/2+\epsilon })$  & $O_\epsilon (|t|^{1/2+\epsilon })$  \\
     $d$ odd  & $O_\epsilon(|t|^{(2d-1)/4+\epsilon })$ & $O_\epsilon (|t|^{3/4+\epsilon })$  \\
      \bottomrule
  \end{tabular}
\end{table}
\noindent
An application of the Phragm\'en--Lindel\"of theorem on the interpolation of growth rates on vertical lines yields the uniform bounds 
\begin{table}[H]
  \centering
  \begin{tabular}{c|c|c}
      \toprule
       & $f$ not cuspidal  & $f$ cuspidal   \\
    \midrule
     $d$ even  & $O_\epsilon (|t|^{(\beta -\sigma )(\tfrac{d-1}{d}+\epsilon)})$  & $O_\epsilon (|t|^{(\beta -\sigma )(\tfrac{1}{d}+\epsilon )})$  \\
     $d$ odd  & $O_\epsilon(|t|^{(\beta -\sigma )(\tfrac{2d-1}{2d}+\epsilon )})$ & $O_\epsilon (|t|^{(\beta -\sigma )(\tfrac{3}{2d}+\epsilon )})$  \\
      \bottomrule
  \end{tabular}
\end{table}
\noindent
for $|R_f(\sigma +it)|$ with $\tfrac{d}{2}\leq \sigma  \leq \beta$ and $|t|\geq1$.
Then each integral along the horizontal edges of the rectangle $\cR$ contributes
\begin{align}\label{eq:horizontal edge}
  \int_{d/2\pm iH}^{\beta \pm iH} R_f(s) \frac{T^s}{s}\, ds \ll \frac{T^\beta }{H}.
\end{align}
Next we will consider the integrals $I_f(H) = \int_1^H |R_f(\tfrac{d}{2}+it)|\, dt$. When $f$ is not cuspidal, an application of the Cauchy--Schwarz inequality together with the integral mean square asymptotics \eqref{asym:zeta}, \eqref{asym:L} and \eqref{zerofree} yields $|I_f(H)|\ll H^{(d+1)/2 + \epsilon}$ for any $\epsilon >0$.
When $f$ is cuspidal, we use \eqref{zerofree}, \cite[Theorem 1]{Zhang2005} and \thmref{thm:sym square mean square} to obtain $|I_f(H)|=O( H^{1+\epsilon})$ if $d$ is even, and $|I_f(H)|=O(H^{5/4+\epsilon})$ if $d$ is odd.  We can conclude that $|I_Q (H)|\ll H^{(d+1)/2 + \epsilon}$. Via a dyadic decomposition we then find that 
\begin{align*}
  \int_1^H \frac{|R_Q (\tfrac{d}{2}+it)|}{|\tfrac{d}{2}+it|}\, dt \ll \sum_{0 \leq k\ll \log H} \frac{2^{k+1}}{H} \int_{H/2^{k+1}}^{H/2^k} |R_Q (\tfrac{d}{2}+it)|\, dt \ll H^{(d-1)/2 + \epsilon}.
\end{align*}
Similarly, the contribution to the integral of the vertical edge $[d/2-iH,d/2-i]$ is $O(H^{(d-1)/2 + \epsilon} )$.  Hence by applying Cauchy's residue theorem to the rectangle $\cR$, and choosing $\beta = d+\epsilon$, we find that 
\begin{align*}
|\Omega_T| = c_1 T^d + c_2 T^{(d+1)/2} +O\op{\frac{T^{d+\epsilon} }{H} + T^{d/2} H^{(d-1)/2 + \epsilon}}.
\end{align*}
The error term is optimized by choosing $H=T^{d/(d+1)}$ and as such of order $O(T^{d-\tfrac{d}{d+1}+\epsilon })$, which is large enough to overtake the term $T^{(d+1)/2}$ when $d^2\geq 2d+1$. Since $c_2=0$ if $d$ is even, the statement follows.
\end{proof}
\begin{rmk}\label{rmk:GRH}
  If we assume that the GRH holds, then the following stronger asymptotic estimates hold:
  \begin{align*}
    |\Omega_T| = c_Q T^d \begin{dcases}
      (1+ O(T^{-(d+1)/(d+2)+\epsilon})), &\text{ if }d \text{ is even}  ;\\
      (1+ O(T^{-(2d+1)/(2d+7)+\epsilon})) , &\text{ if }d \text{ is odd}.
    \end{dcases}
  \end{align*}
\end{rmk}
A variation around the same argument goes to prove
\begin{thm}\label{thm:sum of P}
  Fix $Q$ as in the introduction, let $d$ be even, and let $P$ be a spherical function of degree $\nu>0$ for $Q$. Then 
  \begin{align*}
    \sum_{x\in \Omega_T} P(x) = O(\|P\vert_{\cE_Q}\|_\infty T^{d/2+\epsilon})
  \end{align*}
  for any $\epsilon >0$.
\end{thm}

\begin{proof} 
  When $\nu>0$ and $d$ is even, we recall that 
    \begin{align*}
        R_\Theta (s) =  \frac{L(s-\tfrac{d-1}{2},\tilde \Theta )}{\zeta (s)L(s-\tfrac{d}{2}+1,\chi ) }, 
    \end{align*} 
    where $\tilde \Theta\in S_{d+2\nu }$ is the Shimura lift of $\Theta$, converges absolutely for $\re(s) > \tfrac{d+1}{2}$ and has an analytic continuation to the half-plane $\re(s)\geq \tfrac{d}{2}$. As in the proof of \thmref{Thm1} we find that
    \begin{align*}
        \int_{d/2\pm iH}^{\beta \pm iH}R_\Theta (s) \frac{T^s}{s} ds \ll \frac{T^\beta }{H},
    \end{align*}
   and \propref{prop:ms-recipro} together with \cite[Theorem 1]{Zhang2005} yield 
    \begin{align*}
      \int_{1}^H |R_\Theta(\tfrac{d}{2}+it)|\, dt \ll_N\, \left(\tfrac{(4\pi)^{d+2\nu}}{\Gamma (d+2\nu )}\scal{\tilde \Theta ,\tilde \Theta }\right)^{1/2} H^{1+ \epsilon } 
    \end{align*}
    for any $\epsilon >0$. The argument given to prove \lemref{lm:ThetaNorm} provides us with the estimate $\scal{\tilde\Theta,\tilde \Theta }\ll_{Q,N} \tfrac{\Gamma (d+2\nu )}{(4\pi )^{d+2 \nu }}\|P\vert_{\cE_Q}\|^2_\infty$. Hence
    \begin{align*}
      \int_{1}^{H} & \frac{|R_\Theta (\tfrac{d}{2}+it)|}{|\tfrac{d}{2}+it|} dt
      \ll  \sum_{k\ll \log H} \frac{2^{k+1}}{H} \int_{H/2^{k+1}}^{H/2^k} |R_\Theta (\tfrac{d}{2}+it)| dt  \ll_{Q,N}  H^{\epsilon }\|P\vert_{\cE_Q}\|_\infty.
    \end{align*} 
Upon setting $\beta = \tfrac{d+1}{2} + \epsilon$, we obtain by Cauchy's residue theorem that 
\begin{align*}
  \sum_{x\in \Omega_T} P(x) = O\left(T^{d/2} ( T^{\tfrac{1}{2} + \epsilon} H^{-1} + \|P\vert_{\cE_Q}\|_\infty  H^{\epsilon } )\right).
\end{align*}
The statement follows by choosing $H=T^{1/2}$.
\end{proof}

\bibliographystyle{alpha}
\bibliography{biblio}

  \end{document}